\documentclass{article}
\usepackage[utf8]{inputenc}

\usepackage{amsmath}
\usepackage{amssymb}
\usepackage{amsthm}
\usepackage{blindtext}
\usepackage{tikz}
\usepackage{mathtools}

\usepackage{hyperref}

\newtheorem{theorem}{Theorem}
\newtheorem{definition}[theorem]{Definition}
\newtheorem{corollary}[theorem]{Corollary}
\newtheorem{lemma}[theorem]{Lemma}
\newtheorem{proposition}[theorem]{Proposition}

\newcommand{\V}{\mathcal{V}}
\newcommand{\W}{\mathcal{W}}

\newcommand{\A}{\mathbf{A}}

\newcommand{\Con}{\mathbf{Con }}
\newcommand{\Cg}{\mathrm{Cg}}
\newcommand{\R}{\mathbf{R}}
\newcommand{\U}{\mathcal{U}}

\newcommand{\Free}{\mathbf{F}}

\newcommand{\B}{\mathbf{B}}
\newcommand{\N}{\mathbb{N}}

\newcommand{\Z}{\mathbb{Z}}

\newcommand{\unaryoperations}{r_1,\ldots,r_n}
\newcommand{\rbf}{\mathbf{r}_1,\ldots,\mathbf{r}_n}
\newcommand{\zees}{z,\ldots,z,z}
\newcommand{\exes}{x_1,\ldots,x_k}

\newcommand{\yous}{u_1,\ldots,u_\ell}
\newcommand{\id}{{id}}

\newcommand{\byous}{\mathbf{u}_1,\ldots,\mathbf{u}_\ell}
\newcommand{\bm}{\mathbf{M}}
\newcommand{\bt}{\mathbf{T}}
\newcommand{\bv}{\mathbf{V}}

\newcommand{\Eq}{\mathrm{Eq}}
\newcommand{\fuil}{\mathrm{fi}}
\author{Mateo Muro}
\title{Characterizing Finitely Based Abelian Mal'cev Varieties}
\date{\today}

\begin{document}
\maketitle
\begin{abstract}
    In this paper, we prove the following characterization: an abelian Mal'cev variety is finitely based if and only it has finite type, its ring of idempotent binary terms is finitely presented, and its module of unary terms is finitely presented.
\end{abstract}
\section{Introduction}
In \cite[Theorem 14.9]{MR909290}, Freese and McKenzie proved that every locally finite abelian Mal'cev variety is finitely based. We also know that for an abelian Mal'cev variety $\V$ there exists a ring and module structure on the idempotent binary terms and unary terms in the type of $\V$. We extend these ideas to characterize the finitely based abelian Mal'cev varieties as the abelian Mal'cev varieties with finitely presented rings and modules.

We refer to \cite{MR3793673} for notation and background on general algebras.
A \emph{type} $\sigma$ is a set of function symbols $f$, each associated with a non-negative integer called the \emph{arity of $f$.}
An \emph{algebra} $\A$ of type $\sigma$ is an ordered pair $\langle A, (f^\A\ |\ f\in\sigma)\rangle$. Here, $A$ is a non-empty set, called the \emph{universe} of $\A$. Each $n$-ary $f\in\sigma$ is interpreted as an \emph{$n$-ary basic operation} $f^\A\colon A^n\to A$.
To simplify the notation, we will write $f$ for $f^\A$ if the meaning is clear from the context.

More generally, any term $t$ over $\sigma$ in variables $x_1,\ldots,x_n$ induces an \emph{$n$-ary term operation} $t^\A$ on $A$.
An \emph{equation} over $\sigma$ is an expression of the form $s\approx t$ where $s,t$ are terms over $\sigma.$
We say \emph{$\A$ satisfies $s\approx t$} (written $\A\models s\approx t$) if $s^\A=t^\A$.

A \emph{variety} $\V$ is a class of algebras over a fixed type $\sigma$ that is defined by a set of equations $\Sigma$. Then $\Sigma$ is a basis for $\V$. We say that
$\V$ is \emph{finitely based} if there exists a finite basis for $\V$.

   We say that a variety $\V$ is \emph{Mal'cev} if there exists a term $m(x,y,z)$ in the type of $\V$ such that the following equations hold in $\V$:
    \begin{align*}
        m(x,x,y)\approx y \approx m(y,x,x)
    \end{align*}
    Then, we refer to $m(x,y,z)$ as the \emph{Mal'cev term} of $\V$.
    
    We refer to \cite{MR909290} for definitions and background of commutators of congruences.
    We say that an algebra $\A$ is \emph{abelian} if $[1_\A,1_\A]=0_\A$ where $0_\A,1_\A$ are the trivial and total congruences of $\A$, respectively. A variety $\V$ is abelian if every $\A\in\V$ is abelian. 

    Let $\A$ be an algebra of type $\sigma$ and let $s(x_1,\ldots,x_n),t(x_1,\ldots,x_k)$ be terms over $\sigma$. We say $s$ and $t$ commute in $\A$ if and only if
    \begin{align*}
        \A\models s\bigl(&t(x_{1,1},x_{1,2},\ldots,t_{x_1,k}),\ldots,t(x_{n,1},x_{n,2},\ldots,x_{n,k})\bigl)\\
        &\approx t\bigl(s(x_{1,1},x_{2,1},\ldots,x_{k,1}),\ldots,s(x_{1,k},x_{2,k},\ldots,x_{n,k})\bigl).
    \end{align*}

    \begin{proposition}\label{ab mal alg} \cite[Proposition 5.7]{MR909290}
        An algebra $\A$ with Mal'cev term $m$ is abelian if and only if $m$ commutes with every basic operation of $\A$. 
        \end{proposition}

    For a variety $\V$ and $k\in\mathbb{N}$, let $\Free_\V(\exes,z)$ be the \emph{free algebra in $\V$ over variables $\exes,z$.} Let $F_\V(\exes,z)$ denote the universe of $\Free_\V(\exes,z)$. Then $F_\V(\exes,z)$ represents the $(k+1)$-ary term functions over $\V$ and 
\begin{align*}
    F_\V^{\id}(\exes,z):=\{t\in F_\V(\exes,z)\ |\ t(\zees)= z\}
\end{align*}
represents the $(k+1)$-ary idempotent term functions over $\V$.
 
 For an abelian Mal'cev variety $\V$, we define a ring structure on $F_\V^\id(x,z)$ as on \cite[p. 82]{MR909290}. Later we will prove the following lemma in Section \ref{r&m structure}:

\begin{lemma}\label{small fring}\cite[p.82]{MR909290}
    Let $\V$ be an abelian variety with Mal'cev term $m$. For $s,t\in F_\V(x,z)$ define
    \begin{align*}
         s+t&:=m\bigl(s(x,z),z,t(x,z)\bigl),\\
         s\cdot t&:= s(t(x,z),z)\\
	-s&:= m(z,s(x,z),z).
     \end{align*}
     \begin{enumerate}
     
         \item  $\R_\V:=\langle F_\V^\id(x,z),+,-,\cdot\rangle$ is a ring with identity $x$ and zero $z$.
         
         \item  $\bm_\V:=\langle F_\V(z),+,-,R_\V\rangle$ is an $\R_\V$-module with zero $z$.
         
         \item \label{direct sum} $\langle F_\V(x,z),+,-,R_\V\rangle$ is an $\R_\V$-module isomorphic to the direct sum of the regular $\R_\V$-module and $\bm_\V$.
                 
     	\item\label{ops are ops} Let $a,b,c\in F_\V(x,z)$ and $r\in F_\V^\id(x,z)$. Then $$m(a,b,c)=a-b+c$$ and $$r(a,b)=r\cdot a+(x-r)\cdot b.$$
        
     \end{enumerate}
\end{lemma}

With that, we have enough to state the main theorem.

\begin{theorem}\label{the theorem}
    Let $\V$ be an abelian Mal'cev variety. Then $\V$ is finitely based if and only if
    \begin{enumerate}
        \item $\V$ has finite type,
        \item the ring $\R_\V$ of binary idempotent terms is finitely presented,
        \item and the $\R_\V$-module $\bm_\V$ of unary terms is finitely presented.
    \end{enumerate}
\end{theorem}

For example, the variety of abelian groups has a ring of binary idempotent terms isomorphic to $\Z$ and has $\Z$-module of unary terms isomorphic to $\Z$, both of which are finitely presented, in fact free. Hence, our result proves the already known fact that the variety of abelian groups is finitely based \cite{MR215899}.
In general, if $\R$ is a ring and $\V$ is the variety of $\R$-modules, the ring $\R_\V$ of idempotent binary terms will be isomorphic to $\R$. The result also gives examples of non-finitely based varieties. Consider the ring $\R=\mathbb{Z}[x,y]$ with non-commuting variables $x,y$. Take the ideal $I$ generated by $\{xy^nx\ |\  n\in\mathbb{N}\}$. Then $I$ is not finitely generated and $\R/I$ is not finitely presented. The variety of $\R/I$-modules is equivalent to the variety $\V$ with basic operations $+, -, 0$ and scaling by $x$ and $y$. Since $\R_\V$ is isomorphic to $\R/I$, Theorem \ref{the theorem} yields that $\V$ (and the variety of $\R/I$-modules) is non-finitely based.

We will prove Theorem \ref{the theorem} in Section \ref{final proof}.

\section{The Type of an Abelian Variety}

In this section we show that abelian Mal'cev varieties are essentialy determind by their binary term functions. This will be used later in the paper.

We begin with a lemma that simplifies proving identities in an abelian Mal'cev variety.

\begin{lemma}\label{basis lemma}
     Let $\V$ be an abelian variety of type $\sigma$ with Mal'cev term $m$. Let $s(\exes,z),t(\exes,z)$ be $(k+1)$-ary terms over $\sigma$. Define $s_0(z):=s(z,\ldots,z).$ For $1\leq i\leq k$, define $s_i(x,z):=s(z,\ldots,z,x,z,\ldots,z),$ where $x$ appears in the $i$-th position.
     \begin{enumerate}
         \item\label{malcev decomp}  Then \begin{align*}
             \V&\models t(x_1,\ldots,x_k,z)\approx\\ &m\biggl(\ldots m\Bigl(m\bigl(t_1(x_1,z),t_0(z),t_2(x_2,z)\bigl),t_0(z),t_3(x_3,z)\Bigl),\ldots,t_0(z),t_k(x_k,z)\biggl)
         \end{align*}
         \item\label{easy equiv} Also, $\V\models s(\exes,z)\approx t(\exes,z)$ if and only if $\V\models s_i\approx t_i$ for all $0\leq i\leq k$. 
     \end{enumerate}
 \end{lemma}
 \begin{proof}
    We prove the first item by induction on $k$.
    The base case of $k=0$ is trivial.
    Now assume the claim has been proven for $n$-ary terms.
    Since $m$ is a Mal'cev term and commutes with every term operation by Proposition \ref{ab mal alg}, the variety $\V$ satisfies
    \begin{align*}
        t(x_1,\ldots,x_n,z)&\approx 
        t\bigl(m(x_1,z,z),m(x_2,z,z),\ldots,m(x_{n-1},z,z),m(z,z,x_n),m(z,z,z)\bigl)\\
        &\approx m\bigl(t(x_1,\ldots,x_{n-1},z,z)t(\zees),t(z,z,\ldots,z,x_n,z)\bigl)\\
        &= m\bigl(t(x_1,\ldots,x_{n-1},z,z),t_0(z),t_n(x_n,z)\bigl)\\
        &\approx m\biggl(\ldots m\Bigl(m\bigl(t_1(x_1,z),t_0(z),t_2(x_2,z)\bigl),t_0(z),t_3(x_3,z)\Bigl),\ldots,t_0(z),t_n(x_n,z)\biggl).
    \end{align*}
    The last equation is by the inductive hypothesis. So we have proven the first item.
    
    The second item follows immediately.
 \end{proof}

We use the definition of equivalent varieties from \cite{MR3793673}. Let $\V$ and $\W$ be varieties of respective types $\sigma$ and $\rho$. By an \emph{interpretation} of $\V$ in $\W$ is meant a mapping $D$ from $\sigma$ to the set of terms over $\rho$ such that:
\begin{enumerate}
    \item For $f\in \sigma$, if $f$ is an $n$-ary basic operation for $n>0$, then $D(f)=:f_D$ is an $n$-ary $\rho$-term.
    \item For $f\in \sigma$, if $f$ is a 0-ary basic operation, then $D(f)=:f_D$ is a 1-ary $\rho$-term such that the equation $f_D(x_1)\approx f_D(x_2)$ holds in $\W$.
    \item For every algebra $\A\in\W$, the algebra $\A^D:=\langle A, f^\A_D(f\in \sigma)\rangle$ is in $\V$.
\end{enumerate}

    We say two varieties $\V,\W$ are \emph{equivalent} if there exists two interpretations, $D$ of $\V$ in $\W$ and $E$ of $\W$ in $\V$ such that $\A^{DE}=\A$ for all $A\in\W$ and $\mathbf{B}^{ED}=\mathbf{B}$ for all $\mathbf{B}\in\V$.

    It will be convenient to find normal forms for terms in an abelian Mal'cev variety. Then we can consider the equivalent variety over the normal forms instead. We first define a variety $\U$ such that every abelian Mal'cev variety of finite type is equivalent to a subvariety of $\U$.

 \begin{definition}\label{C}
 Let $\U$ be the variety of type $\{\yous,\unaryoperations,m\}$ where $\ell\leq n $, $\yous$ are unary operations, $\unaryoperations$ are binary operations, and $m$ is a ternary operation defined by
 \begin{enumerate}
     \item \label{U1} $m(x,z,z)\approx m(z,z,x)\approx x$,
     \item \label{U2}$m\bigl(m(x_1,x_2,x_3),m(y_1,y_2,y_3),m(z_1,z_2,z_3)\bigl)\approx m\bigl(m(x_1,y_1,z_1),m(x_2,y_2,z_2),m(x_3,y_3,z_3)\bigl)$
     \item \label{U3} $r_i(z,z)\approx z$ for all $i\leq n $,
     \item \label{U4} $m\bigl(r_i(x_1,x_2),r_i(y_1,y_2),r_i(z_1,z_2)\bigl)\approx r_i\bigl(m(x_1,y_1,z_1),m(x_2,y_2,z_2)\bigl)$ for all $i \leq n$,
     \item \label{U5} $m\bigl(u_i(x),u_i(y),u_i(z)\bigl)\approx u_i\bigl(m(x,y,z)\bigl)$ for all $i\leq\ell,$
     \item \label{U6} $m(u_i(x),u_i(z),z)\approx r_i(x,z)$ for all $i\leq\ell.$
 \end{enumerate}
\end{definition}

 Now we prove that every term in an abelian Mal'cev variety decomposes into a composition of the Mal'cev term, binary idempotent terms, and unary terms. Said decomposition gives an equivalent variety $\W$ contained in $\U$ when the type is finite.
 
 \begin{lemma}\label{term equivalence}
     Let $\V$ be an abelian variety of type $\sigma$ with Mal'cev term $m$.
     \begin{enumerate}
         \item Then $\V$ is equivalent to a variety $\W$ of type $F_\V(z)\cup F_\V^\id(x,z)\cup\{m\}.$
         \item If $\sigma$ is infinite, then $\V$ cannot be finitely based.
         \item If $\sigma$ is finite, then there exists $\ell\leq n\in\mathbb{N}$ and $\yous\in F_\V(z)$, $\unaryoperations\in F_\V^\id(x,z)$ such that $\V$ is equivalent to a subvariety $\W$ of $\U$ of type $\{\yous,\unaryoperations,m\}$ as in Definition \ref{C}.
     \end{enumerate}
 \end{lemma}
 
 \begin{proof}
     1. Let $\W$ be a variety of type $\rho:=F_\V(z)\cup F_\V^\id(x,z)\cup\{m\}.$
     Define $E$ to be an interpretation of $\W$ in $\V$ as follows:
     For $f\in\rho$, let $E(f)$ be a term over $\sigma$ that induces $f$ in $\Free_\V(x,y,z).$
     For $\A\in\V$, let $\A^E$ be the algebra with universe $A$, unary operations $u^\A(z)$ for $u\in F_\V(z),$ binary operations $r^\A(x,z)$ for $r\in F_\V^\id(x,z),$ and ternary operation $m^\A(x,y,z)$ for the Mal'cev term $m$ of $\V$.
     Here we slightly abuse notation when writing $u^\A$  for the term function on $\A$ that is induced by the term $u$ over $\sigma$ which induces the element $u(z)\in F_\V(z)$. Similarly for $r^\A$.
     Then $E$ is an interpretation of $\V$ in $\W:=\{\A^E\ |\ \A\in\V\}$.
     
     Conversely, define $D$ to be an interpretation of $\V$ in $\W$ as follows:
     Let $f\in \sigma$.
     In the case that $f$ is 0-ary, let $t(z)\in F_\V(z)$ be any term such that $t(z)=f$ in $F_\V(z)$.
     In the case that $f(x_1,\ldots,x_k,z)$ is $(k+1)$-ary, define $f_0(z):=f(z,\ldots,z)$ and $f_i(x):=f(z,\ldots,z,x,z,\ldots,z)$, where $x$ appears in the $i$-th position. Then define $$D(f):=m\biggl(\ldots m\Bigl(m\bigl(t_1(x_1,z),t_0(z),t_2(x_2,z)\bigl),t_0(z),t_3(x_3,z)\Bigl),\ldots,t_0(z),t_k(x_k,z)\biggl).$$

     Then $D$ is an interpretation of $\W$ in $\V$ by Lemma \ref{basis lemma}.\ref{malcev decomp}. Furthermore, for any algebra $\A^E\in\W$, the algebra $\A^{ED}=\A$ and hence in $\V$.

    It follows that for all $\A\in\V$, we have $\A^{ED}=\A\in\V$ and $\A^{EDE}=\A^E\in\W$. The latter means that $\B^{DE}=\B$ for all $\B\in\W.$ Thus $\V$ and $\W$ are equivalent.

     2. Suppose that $\V$ has infinite type and suppose by way of contradiction that $B$ is a finite basis for $\V$. Then there exists a basic operation $f\in\sigma$ such that $f$ does not occur in $B$. Say $f$ has arity $k$. Since $\V$ is abelian Mal'cev, by Proposition \ref{ab mal alg}, we know that \begin{align*}
         \V\models  &\ m\bigl(f(x_1,\ldots,x_k),f(y_1,\ldots,y_k),f(z_1,\ldots,z_k)\bigl)\\&\approx f\bigl(m(x_1,y_1,z_1),\ldots,m(x_k,y_k,z_k)\bigl).
     \end{align*} However, $B$ cannot prove this fact since $f$ never occurs in $B$, contradicting $B$ being a basis.
     
     3. If $\sigma$ is finite, you need finitely many unary terms and idempotent binary terms with $m$ to interpret every basic operation in $\sigma$. Let us denote said unary terms as $u_1,\ldots,u_\ell$, where $\ell\in\mathbb{N}$. Now define $r_i(x,z)=m(u_i(x),u_i(z),z)$ for $i\leq \ell$. Let $r_i(x,z)$ be the rest of the needed idempotent binary terms for $\ell<i\leq n$, where $n\geq\ell$. Then we need to show that the terms $u_1,\ldots,u_\ell,r_1,\ldots,r_n$ satisfy Definition \ref{C}.\ref{U1}-\ref{U6}. That would complete the proof.

     Equation \ref{U1} is satisfied because $m$ is assumed to be a Mal'cev term. Equations \ref{U2},\ref{U4}, and \ref{U5} are satisfied by Proposition \ref{ab mal alg}. Equation \ref{U3} is satisfied since $r_i\in F_\V^\id(x,z)$ for $i\leq n$. Lastly, Equation \ref{U6} holds by construction.
 \end{proof}

\section{Ring and Module Structure}\label{r&m structure}

We now turn our attention to the ring and module structures in abelian Mal'cev varieties. We first prove Lemma \ref{small fring}.

\begin{proof}[Proof of Lemma \ref{small fring}]
    We prove that $\R_\V$ is a ring and that $\bm_\V$ is an $\R_\V$-module by collecting proofs of properties of $+,-,\cdot$. \cite[Lemma 5.6]{MR909290} already shows that $\langle F_\V(x,z),+,-,z\rangle$ is an abelian group. Let $a,b,c\in F_\V(x,z)$ and $r,s\in F_\V^\id(x,z).$
    
    Now we show that $\langle F_\V^\id(x,z), +,-\rangle$ is a subgroup. Clearly, $z\in F_\V^\id(x,z)$. We have
    \begin{align*}
        r(z,z)-s(z,z)=z-z=z
    \end{align*}
    So $r-s\in F_\V^\id(x,z)$ and hence $\langle F_\V^\id(x,z),+,-,z\rangle$ is an abelian group.
    
   Multiplication is associative since
     \begin{align*}
         (a\cdot b)\cdot c= a\Bigl(b\bigl(c(x,z),z\bigl)\Bigl)=a\cdot(b\cdot c).
     \end{align*}
     
     We show now that multiplication distributes over addition if the left factor is in $F_\V^\id(x,z)$. We have, since $r$ is idempotent and commutes with $m$ by Proposition \ref{ab mal alg}, that
    \begin{align*}
        r(a+b)&=r\bigl(m(a,z,b),z\bigl)\\
        &= r\bigl(m(a,z,b),m(z,z,z)\bigl)\\
        &= m\bigl(r(a,z),r(z,z),r(b,z)\bigl)\\
        &= m\bigl(r(a,z),z,r(b,z)\bigl)\\
        &=ra+rb.
    \end{align*}
    
    Now we show right distributivity. We have
    \begin{align*}
        (a+b)\cdot c&= m\Bigl(a\bigl(c(x,z),z\bigl),z,b\bigl(c(x,z),z\bigl)\Bigl)\\
        &=a\cdot c +b\cdot c.
    \end{align*}

    Clearly $1(x,z):=x$ is a multiplicative identity. So $\R_\V$ is a ring and $F_\V(x,z)$ forms an $\R_\V$-module with submodules $F_\V^\id(x,z)$ and $F_\V(z)$. To see \ref{direct sum}, note that every element $a$ in $F_\V(x,z)$ can be written uniquely as a sum of an idempotent term and a unary term in the following way
    \begin{align*}
        a(x,z)&= \bigl(a(x,z)-a(z,z)\bigl)+a(z,z).
    \end{align*}

    \ref{ops are ops}. We have by \cite[Lemma 5.6]{MR909290} that $m(a,b,c)=a-b+c.$ Let $a,b\in F_\V(x,z)$ and $r\in F_\V^\id(x,z)$. 
    Then, we have, using the idempotence of $r$,
	\begin{align*}
	r(a,b)
        &= r\bigl( m(a,b,b), m(z,z,b) \bigl)\\
        &= m\bigl( r(a,z),r(b,z),r(b,b) \bigl)\\
        &= m\bigl( r(a,z),r(b,z),b \bigl)\\
        &= r\cdot a-r\cdot b+b\\
        &= r\cdot a+(x-r)\cdot b.
    \end{align*}
    
\end{proof}

\begin{lemma}\label{ab mal algebra unique}
    Let $\V$ be an abelian Mal'cev variety with Mal'cev terms $m$ and $m'$. Then $\V\models m(x,y,z)\approx m'(x,y,z)$.
\end{lemma}
\begin{proof}
    Using that $m$ and $m'$ commute by Proposition \ref{ab mal alg}, we have
    \begin{align*}
        \V\models m(x,y,z)&\approx m\bigl(m'(x,y,y),m'(y,y,y),m'(y,y,z)\bigl)\\
            &\approx m'\bigl(m(x,y,y),m(y,y,y),m(y,y,z)\bigl)\\
            &\approx m'(x,y,z).
    \end{align*}
    Hence, the two terms induce the same term operation in $\V$.
\end{proof}

We show that for the variety $\U$ as in Definition \ref{C}, the ring $\R_\U$ and the $\R_\U$-module $\bm_\U$ are free. Further, we characterize finitely generated fully invariant congruences of $\Free_\U(x,z)$ in terms of ideals and submodules of $\bm_\U$.

 \begin{lemma} \label{Fring}
    Let $\U$ be as in Definition \ref{C}.
     \begin{enumerate}

        \item\label{U ab mal} $\U$ is an abelian Mal'cev variety.
 
        \item \label{free ring}   The ring $\R_\U$ is free over generators $\unaryoperations.$

        \item\label{free submodule} The $\R_\U$-module $\bm_U$ is free over generators $u_1,\ldots,$ $u_\ell.$

        \item \label{U con is module con} A congruence $\theta$ on $\Free_\U(x,z)$ is fully invariant if and only if
        \begin{enumerate}
            \item \label{I} $I:=z/\theta\cap R_\U$ is an ideal of $\R_\U$,        
            \item \label{N} $N:=z/\theta\cap M_\U$ is an $\R_\U$-submodule of $\bm_\U$,        
            \item \label{I+N} $z/\theta=I+N$,
            \item \label{IM} $IM_\U\subseteq N$,
            \item \label{f(x)-f(z)} and $f(x)-f(z)-x\in I$ for all $f(z)\in N$.
        \end{enumerate}

        \item \label{finite congruence} A fully invariant congruence $\theta$ on $\Free_\U(x,z)$ is finitely generated if and only if 
        \begin{enumerate}
            \item the ideal $I:=z/\theta\cap R_\U$ of $\R_\U$ is finitely generated and
            \item for $N:=z/\theta\cap M_\U$ the $\R_\U$-module $N/IM_\U$ is finitely generated.
        \end{enumerate}
     \end{enumerate}
 \end{lemma}
 \begin{proof}

\ref{U ab mal}. Every basic operation commutes with the Mal'cev term $m$. So $\U$ is an abelian Mal'cev variety by Proposition \ref{ab mal alg}.

 \ref{free ring}. Let $\U'$ be the variety of type $\{r_1,\ldots,r_n,m\}$ satisfying identities \ref{U1}-\ref{U4} of $\U$. Note that this makes $\U'$ an idempotent reduct of $\U$ and $+,-$ are still term operations of $\U'$ and $z$ is still the additive identity. Also, $\U$ and $\U'$ are abelian varieties. The unique Mal'cev term is $m(x,y,z)=x-y+z$ by Lemma \ref{ab mal algebra unique}. It commutes with every basic operation by Proposition \ref{ab mal alg}.

First, we claim \begin{align}\label{U' eqn}
F_{\U'}(x,z)=F_\U^\id(x,z).
\end{align}
Clearly, $F_{\U'}(x,z)\subseteq F_\U^\id(x,z)$. All that is left is to show $F_{\U'}(x,z)\supseteq F_\U^\id(x,z)$. Note that every $s\in F_\U^\id(x,z)$ has the form $$t(x,z)-t(z,z)$$ for some $t\in F_\U(x,z)$. We then induct on $t$ to show that $s\in F_{\U'}(x,z)$.

For the first base case, suppose $t=x$. Then $s=x-z=x\in F_{\U'}(x,z)$. Now, as a second base case, suppose $t=z$. Then $s=z-z=z\in F_{\U'}(x,z)$.

For the induction step, we consider several cases. First, suppose that $t=u_i(v(x,z))$ for some $i\leq\ell$ and some $v\in F_\U(x,z)$. 
 We have
\begin{align*}
	s(x,z) & = u_i\bigl(v(x,z)\bigl)-u_i\bigl(v(z,z)\bigl)\\ 
		&=u_i\bigl(v(x,z)\bigl)-u_i\bigl(v(z,z)\bigl)+u_i(z)-u_i(z)\\
		&=u_i\bigl(v(x,z)-v(z,z)+z\bigl)-u_i(z)\\
		&=r_i\bigl(v(x,z)-v(z,z)\bigl),\\
\end{align*}
 which is in $F_{\U'}(x,z)$ since $v(x,z)-v(z,z)\in F_{\U'}(x,z)$ by the inductive hypothesis.

For another induction case, suppose that $t=r_i(v,w)$ for some $i\leq n $ and some $v,w\in F_\U(x,z)$. 
Then
\begin{align*}
	s(x,z) &= r_i\bigl(v(x,z),w(x,z)\bigl)-r_i\bigl(v(z,z),w(z,z)\bigl)\\
        &= r_i\bigl(v(x,z),w(x,z)\bigl)-r_i\bigl(v(z,z),w(z,z)\bigl)+r_i(z,z)\text{ since } r_i \text{ is idempotent}\\
		&=r_i\bigl(v(x,z)-v(z,z)+z,w(x,z)-w(z,z)+z\bigl),
\end{align*}
which is in $F_{\U'}(x,z)$ since $v(x,z)-v(z,z),w(x,z)-w(z,z)\in F_{\U'}(x,z)$ by the inductive hypothesis.

For the last induction step, suppose that $t=m(v,w,y)$ for some $v,w,y\in F_\U(x,z)$. 
Using Lemma \ref{small fring}.\ref{ops are ops}, we have
\begin{align*}
	s(x,z)&=\bigl(v(x,z)-w(x,z)+y(x,z)\bigl)-\bigl(v(z,z)-w(z,z)+y(z,z)\bigl)\\
    &=\bigl(v(x,z)-v(z,z)\bigl)-\bigl(w(x,z)-w(z,z)\bigl)+\bigl(y(x,z)-y(z,z)\bigl)\\
    &=m\bigl(v(x,z)-v(z,z),w(x,z)-w(z,z),y(x,z)-y(z,z)\bigl),
\end{align*}
which is in $F_{\U'}(x,z)$ since $v(x,z)-v(z,z),w(x,z)-w(z,z),y(x,z)-y(z,z)\in F_{\U'}(x,z)$ by inductive hypothesis.

So we have shown $F_{\U'}(x,z)=F_\U^\id(x,z)$. It follows that $\R_\U=\R_{\U'}.$

Now we claim
\begin{align}\label{T eqn}
    \R_{\U'} \cong \langle T,+,\cdot \rangle, \text{ the free ring over generators } \rbf.
\end{align} 
We will write elements of $T$ in boldface to better distinguish them from elements of $F_{\U'}(x,z).$

Clearly, $$\psi\colon T\to F_{\U'}(x,z),\ \mathbf{r}_i\mapsto r_i\text{ for } i\leq n, $$ defines a ring homomorphism. To show $\psi$ is bijective, we will make $T$ an algebra in $\U'$. Then we will construct a homomorphism $$\tau\colon F_{\U'}(x,z)\to T$$ such that $\tau\circ\psi$ is the identity on $T$ and $\psi\circ\tau$ is the identity on $F_{\U'}(x,z)$.
We first have to define operations $r_1^\bt,\ldots,r_n^\bt,m^\bt$ on $T$ such that $\bt:=\langle T,r_1^\bt,\ldots,r_n^\bt,m^\bt\rangle\in\U'$. 
For $\mathbf{s}_1,\mathbf{s}_2,\mathbf{s}_3\in T$ and $i\leq n $ define
\begin{align*}
	r_i^\bt(\mathbf{s}_1,\mathbf{s}_2):&=\mathbf{r}_i\mathbf{s}_1+(\mathbf{1}-\mathbf{r}_i)\mathbf{s}_2,\\
	m^\bt(\mathbf{s}_1,\mathbf{s}_2,\mathbf{s}_3):&=\mathbf{s}_1-\mathbf{s}_2+\mathbf{s}_3.
\end{align*}

From the definition, it is clear that $m^\bt$ is a Mal'cev operation that commutes with itself and the idempotent affine operations $r_i^\bt$ for all $i\leq n $. Hence $\bt$ is in $\U'$.

Now we show that $\bt$ and $\Free_{\U'}(x,z)$ are isomorphic. Let $\tau\colon\Free_{\U'}(x,z)\to\bt$ be the homomorphism of algebras in $\U'$ defined by $x\mapsto\mathbf{1}$ and $z\mapsto\mathbf{0}$. We show that $\tau\circ\psi$ is the identity map on $T$ and that $\psi\circ\tau$ is the identity map on $F_{\U'}(x,z).$

We show that \begin{equation}\label{taupsi}
  \tau(\psi(\mathbf{s}))=\mathbf{s}  \text{ for all } \mathbf{s}\in T
\end{equation}  via induction on $\mathbf{s}$. 
    For the base case, we have $\tau(\psi(\mathbf{0}))=\tau(z)=\mathbf{0}$ and $\tau(\psi(\mathbf{1}))=\tau(x)=\mathbf{1}$.
    For $i\leq n $, we have \begin{align*}
     \tau(\psi(\mathbf{r}_i))&=\tau(r_i(x,z))\\
     &=r_i^\bt(\tau(x),\tau(z))\\
     &=r_i^\bt(\mathbf{1},\mathbf{0})\\
     &=\mathbf{r}_i\mathbf{1}+(\mathbf{1}-\mathbf{r}_i)\mathbf{0}\\
     &=\mathbf{r}_i.   
    \end{align*}

 For the induction step, let $\mathbf{s}_1,\mathbf{s}_2\in T$ such that $\tau(\psi(\mathbf{s}_i))=\mathbf{s}_i$ for $i\leq 2$. 
 Then
 \begin{align*}
    \tau\bigl(\psi(\mathbf{s}_1+\mathbf{s}_2)\bigl)&=\tau\bigl(\psi(\mathbf{s}_1)+\psi(\mathbf{s}_2)\bigl)\\
        &=\tau\Bigl(m\bigl(\psi(\mathbf{s}_1),z,\psi(\mathbf{s}_2)\bigl)\Bigl)\\
        &=m^\bt\Bigl(\tau\bigl(\psi(\mathbf{s}_1)\bigl),\tau(z),\tau\bigl(\psi(\mathbf{s}_2)\bigl)\Bigl)\\
        &=m^\bt(\mathbf{s}_1,\mathbf{0},\mathbf{s}_2)\\
        &=\mathbf{s}_1-\mathbf{0}+\mathbf{s}_2\\
        &=\mathbf{s}_1+\mathbf{s}_2.
 \end{align*}
 
 Also,
 \begin{align*}
     \tau\bigl(\psi(\mathbf{s}_1\mathbf{s}_2)\bigl)&=\tau\bigl(\psi(\mathbf{s}_1)\psi(\mathbf{s}_2)\bigl)\\
     &=\tau\bigl(\psi(\mathbf{s}_1)(\psi(\mathbf{s}_2),z)\bigl)\\
     &=\tau\bigl(\psi(\mathbf{s}_1)\bigl)\Bigl(\tau\bigl(\psi(\mathbf{s}_2)\bigl),\tau(z)\Bigl)\\
     &=\mathbf{s}_1(\mathbf{s}_2,\mathbf{0})\\
     &=\mathbf{s}_1\mathbf{s}_2+(\mathbf{1}-\mathbf{s}_1)\mathbf{0}\\
     &=\mathbf{s}_1\mathbf{s}_2.
 \end{align*}
So we have shown that $\tau\circ\psi$ is the identity map on $T$.

Conversely, we show that \begin{equation}\label{psitau}
    \psi(\tau(s))=s \text{ for all } s\in F_{\U'}(x,z)
\end{equation}   via induction on $s$.
    For the base cases, we have $\psi(\tau(x))=\psi(\mathbf{1})=x$ and $\psi(\tau(z))=\psi(\mathbf{0})=z$.

Let $s_1,s_2,s_3\in F_{\U'}(x,z)$ such that $\psi(\tau(s_j))=s_j$ for $j\leq 3.$ 
Then
\begin{align*}
    \psi\Bigl(\tau\bigl(m(s_1,s_2,s_3)\bigl)\Bigl)&=\psi\Bigl(m^\bt\bigl(\tau(s_1),\tau(s_2),\tau(s_3)\bigl)\Bigl)\\
    &=\psi\bigl(\tau(s_1)-\tau(s_2)+\tau(s_3)\bigl)\\
    &=\psi\bigl(\tau(s_1)\bigl)-\psi\bigl(\tau(s_2)\bigl)+\psi\bigl(\tau(s_3)\bigl)\\
    &=s_1-s_2+s_3\\
    &=m(s_1,s_2,s_3)\text{ by Lemma } \ref{small fring}.\ref{ops are ops}.
\end{align*}

Also, for $i\leq n$,
\begin{align*}
    \psi\Bigl(\tau\bigl(r_i(s_1,s_2)\bigl)\Bigl)&=\psi\Bigl(r_i^\bt\bigl(\tau(s_1),\tau(s_2)\bigl)\Bigl)\\
    &=\psi(\mathbf{r}_i\tau(s_1)+(\mathbf{1}-\mathbf{r}_i)\tau(s_2))\\
    &=\psi(\mathbf{r}_i)\psi\bigl(\tau(s_1)\bigl)+\psi(\mathbf{1}-\mathbf{r}_i)\psi\bigl(\tau(s_2)\bigl)\\
    &=r_is_1+(1-r_i)s_2\\
    &=r_i(s_1,s_2) \text{ by Lemma } \ref{small fring}.\ref{ops are ops}.
\end{align*}

So we have shown that $\psi\circ\tau$ is the identity on $F_{\U'}(x,z)$. Together with the fact that $\tau\circ\psi$ is the identity on $T$, we have shown (\ref{T eqn}). This, along with (\ref{U' eqn}), proves Lemma \ref{Fring}.\ref{free ring}.

\ref{free submodule}. We want to show \begin{equation}\label{V}
    \bm_\U\cong \langle V,+,R_\U\rangle, \text{ the free } \R_\U\text{-module over generators } \byous.
\end{equation} 
We write elements of $V$ in boldface to better distinguish them from elements in $F_\U(z)$.
Clearly, $\psi\colon V\to F_\U(z)$, $\mathbf{u}_i\mapsto u_i$ for $i\leq\ell$, defines an $\R_\U$-module homomorphism.
To show that $\psi$ is bijective, we will make $V$ an algebra in $\U$. Then we will construct a homomorphism $\tau\colon F_\U(z)\to V$ such that $\tau\circ\psi$ is the identity map on $V$ and $\psi\circ\tau$ is the identity map on $F_\U(z)$.
We then have to define operations $u_1^\bv,\ldots,u_\ell^\bv,r_1^\bv,\ldots,r_n^\bv,m^\bv$ on $V$ such that $\bv:=\langle V,u_1^\bv,\ldots,u_\ell^\bv,r_1^\bv,\ldots,r_n^\bv,m^\bv\rangle$ is in $\U$. 
For $\mathbf{v}_1,\mathbf{v}_2,\mathbf{v}_3\in V$, $i\leq\ell$ and $j\leq n $ define
\begin{align*}
    u_i^\bv(\mathbf{v_1})&:=r_i\mathbf{v_1}+\mathbf{u}_i,\\
	r_j^\bv(\mathbf{v}_1,\mathbf{v}_2):&=r_j\mathbf{v}_1+(1-r_j)\mathbf{v}_2,\\
	m^\bv(\mathbf{v}_1,\mathbf{v}_2,\mathbf{v}_3):&=\mathbf{v}_1-\mathbf{v}_2+\mathbf{v}_3.
\end{align*}

The definition clearly shows that $m^\bv$ is a Mal'cev operation. It is also clear that $m^\bv$ commutes with itself and the idempotent affine operations $r_i^\bv$ for all $i\leq n $ and the unary operations $u_i^\bv$. Hence, $\bv\in \U.$

Now we show that $\bv$ and $\Free_\U(z)$ are isomorphic. Let $\tau\colon\Free_{\U}(z)\to\bv$ be the homomorphism of algebras in $\U$ defined by $z\mapsto\mathbf{0}$. We show that $\tau\circ\psi$ is the identity on $V$ and that $\psi\circ\tau$ is the identity on $F_{\U}(z).$

We show that \begin{equation}\label{taupsiv}
    \tau(\psi(\mathbf{v}))=\mathbf{v} \text{ for all } \mathbf{v}\in V
\end{equation}  via induction on $\mathbf{v}$. 
    For the first base case, we have $\tau(\psi(\mathbf{0}))=\tau(z)=\mathbf{0}$. 
    For $i\leq\ell$ we have
    \begin{align*}
      \tau(\psi(\mathbf{u}_i)      &=\tau(u_i(z))=u_i^\bv(\tau(z))\\
      &=u_i^\bv(\mathbf{0})\\
      &=r_i\mathbf{0}+\mathbf{u}_i\\
      &=\mathbf{u}_i.  
    \end{align*}

For the induction step, let $\mathbf{v},\mathbf{w}\in V$ such that $\tau(\psi(\mathbf{v}))=\mathbf{v}$ and $\tau(\psi(\mathbf{w}))=\mathbf{w}$. 
Then
\begin{align*}
    \tau\bigl(\psi(\mathbf{v}+\mathbf{w})\bigl)&=\tau\bigl(\psi(\mathbf{v})+\psi(\mathbf{w})\bigl)\\
    &=m^\bv\Bigl(\tau\bigl(\psi(\mathbf{v})\bigl),\tau(z),\tau\bigl(\psi(\mathbf{w})\bigl)\Bigl)\\
    &=m^\bv(\mathbf{v},\mathbf{0},\mathbf{w})\\
    &=\mathbf{v}+\mathbf{w}.
\end{align*}
Let $r\in F_\U^\id(x,z)$. 
Then
\begin{align*}
    \tau\bigl(\psi(r\mathbf{v})\bigl)&= \tau\bigl(r\psi(\mathbf{v})\bigl)\\
    &=r^\bv\Bigl(\tau\bigl(\psi(\mathbf{v})\bigl),\tau(z)\Bigl)\\
    &=r^\bv(\mathbf{v},\mathbf{0})\\
    &=r\mathbf{v}+(1-r)\mathbf{0}\\
    &=r\mathbf{v}.
\end{align*}
So we have shown that $\tau\circ\psi$ is the identity map on $V$.

Conversely, we show that \begin{equation}\label{psitauv}
    \psi(\tau(t)=t \text{ for all } t\in F_\U(z)
\end{equation} via induction on $t(z)$.  For the base case we have $\psi(\tau(z))=\psi(\mathbf{0})=z.$  

For the induction steps, let $t_1,t_2,t_3\in F_\U(z)$ such that $\psi(\tau(t_j))=t_j$ for $j\leq 3$. 
Then
\begin{align*}
    \psi\Bigl(\tau\bigl(m(t_1,t_2,t_3)\bigl)\Bigl)&=\psi\Bigl(m^\bv\bigl(\tau(t_1),\tau(t_2),\tau(t_3)\bigl)\Bigl)\\
    &=\psi\bigl(\tau(t_1)-\tau(t_2)+\tau(t_3)\bigl)\\
    &=\psi\bigl(\tau(t_1)\bigl)-\psi\bigl(\tau(t_2)\bigl)+\psi\bigl(\tau(t_3)\bigl)\\
    &=t_1-t_2+t_3\\
    &=m(t_1,t_2,t_3).
\end{align*}
Suppose $i\leq n $. Then
\begin{align*}
\psi\Bigl(\tau\bigl(r_i(t_1,t_2)\bigl)\Bigl)&=\psi\Bigl(r_i^\bv\bigl(\tau(t_1),\tau(t_2)\bigl)\Bigl)\\
&=\psi\bigl(r_i\tau(t_1)+(1-r_i)\tau(t_2)\bigl)\\
&=r_i\psi\bigl(\tau(t_1)\bigl)+(1-r_i)\psi\bigl(\tau(t_2))\\
&=r_it_1+(1-r_i)t_2\\
&=r_i(t_1,t_2)\text{ by Lemma }\ref{small fring}.\ref{ops are ops}.
\end{align*}

Suppose that $i\leq\ell$. 
We have
\begin{align*}
    \psi\Bigl(\tau\bigl(u_i(t)\bigl)\Bigl)&=\psi\Bigl(u_i^\bv\bigl(\tau(t)\bigl)\Bigl)\\
    &=\psi\bigl(r_i\tau(t)+\mathbf{u}_i\bigl)\\
    &=r_i\psi\bigl(\tau(t)\bigl)+\psi(\mathbf{u}_i)\\
    &=r_it+u_i(z)\\
    &=u_i(t)-u_i(z)+u_i(z)\\
    &=u_i(t).
\end{align*}

So we have shown $\psi\circ\tau$ is the identity map on $M_\U$. Together with $\tau\circ\psi$ being the identity map on $V$, this shows (\ref{V}). This proves Lemma \ref{Fring}.\ref{free submodule}.

\ref{U con is module con}.
We first prove the forward direction.
Assume that $\theta$ is a fully invariant congruence of $\Free_\U(x,z)$.
We prove \ref{I}. That is, we prove that $z/\theta\cap R_\U=I$ is an ideal of $\R_\U$.
$I$ is clearly an additive subgroup of $\R_\U$.
To prove $I$ is multiplicatively closed, consider $s\in R_\U$ and $t\in I$.
To see $s\cdot t, t\cdot s\in R_\U$, note that $\R_\U$ is a ring and thus multiplicatively closed.

To prove that $s\cdot t, t\cdot s\in z/\theta$, note that because $\theta$ is a congruence, we have
\begin{align*}
	s(t(x,z),z) \ \theta\ s(z,z)
			  = z.\\
\end{align*}
Since $\theta$ is a fully invariant congruence, we have that $t(x,z)\ \theta\ z$ implies 
\begin{align*}
	t(s(x,z),z) \ \theta\ z.
\end{align*} 
Hence $s\cdot t,t\cdot s\in I$ and $I$ is an ideal. 
This proves \ref{I}.

Now we show \ref{N}. That is, we show that $z/\theta\cap M_\U=N$ is an $\R_\U$-submodule of $\bm_\U$. 
Clearly, $N$ is an additive subgroup.
To show $N$ is closed under the action by $\R_\U$, consider $s\in R_\U$, $v\in N$.
Since $\theta$ is a congruence,
\begin{align*}
	s(v,z) \ \theta\ s(z,z)
		   = z.
\end{align*}
 Hence $s\cdot v\in N$ and $N$ is an $\R_\U$-submodule of $\bm$.
 This proves \ref{N}.

We now show \ref{I+N}. For any $s(x,z)\ \theta\ z$, note that $s(x,z)-s(z,z)\in I$ and $s(z,z)\in N$. Hence $z/\theta=I+N.$
This proves \ref{I+N}.

We now show \ref{IM}, that $IM_\U\subseteq N.$
Let $s\in I$ and $v\in M_\U$.
Then 
\begin{align*}
        s(v,z)&\ \theta\ z\text{ since } \theta \text{ is fully invariant.}                        
\end{align*}
Since $s(v,z)\in M_\U$ and $s(v,z)\in z/\theta$, we have that $s(v,z)\in N$.
So $IM_\U\subseteq N$.
This proves \ref{IM}.

We now show \ref{f(x)-f(z)}. That is, we show that $f(x)-f(z)-x\in I$ for all $f(z)\in N$.
Suppose that $f(z)\in N$.
Then $f(x)\ \theta\ x$ because $\theta$ is fully invariant.
Also $f(x)-f(z)\ \theta\ x-z$, and $f(x)-f(z)-x\ \theta\ z$.
Also, $f(x)-f(z)-x\in R_\U$ so $f(x)-f(z)-x\in I$.
This proves \ref{f(x)-f(z)}.
We have proven Properties \ref{I}-\ref{f(x)-f(z)}.

Now we prove the converse. Suppose that $\theta,I,N$ satisfy properties \ref{I}-\ref{f(x)-f(z)}.
We introduce some notation. 
For $f\in F_\U(x,z)$, define $f^\id(x,z):=f(x,z)-f(z,z)$ and $f^u(z):=f(z,z)$. Note that $f(x,z)=f^\id(x,z)+f^u(z).$

Now we show that $\theta$ is fully invariant.
Let $t(x,z)\in z/\theta$ and $v,w\in F_\U(x,z)$.
We need to show that $t(v,w)\ \theta\ w$. 
Since $t(v,w)=m(t^\id(v,w),w,t^u(w))$ and $m(w,w,w)=w$,
it will suffice to show that $t^\id(v,w)\ \theta\ w\ \theta\ t^u(w)$.

Note that for $r\in I$ and $s\in F_\U(x,z)$, we have
\begin{align*}
    r\cdot s &= r\cdot(s^\id+s^u)\\
            &= \underbrace{r\cdot s^\id}_{\in I}+\underbrace{r\cdot s^u}_{\in IM_\U\subseteq N}\\
            &\ \theta\ z.
\end{align*}

Since $t\in z/\theta$, we have $t^\id\in z/\theta\cap R_\U=I$ and $t^u\in z/\theta\cap M_\U=N$.

We have
\begin{align*}
    t^\id(v,w)&=t^\id\cdot v+(x-t^\id)\cdot w\\
    &= t^\id \cdot v+w-t^\id\cdot w\\
    &\ \theta\ w.
\end{align*}

Since $t^u(z)\in N$, by \ref{f(x)-f(z)}, $t^u(x)-t^u(z)-x\in I$. So we have that
\begin{align*}
    t^u(w)-z-w&\ \theta\ (t^u(x)-t^u(z)-x)\cdot w\\
    &\ \theta\ z.
\end{align*}
This implies that $t^u(w)\ \theta\ w$ and we are done. 

\ref{finite congruence}. We prove the forward direction first. Suppose that $\theta$ is a finitely generated fully invariant congruence of $\Free_\U(x,z).$  
Since $(s,t)\in\theta$ if and only if $(m(s,t,z),z)\in\theta$, we can assume that the generating set for $\theta$ has the form
$$\{(t_1,z),\ldots,(t_k,z)\}$$ for some $k\in\mathbb{N}$ and $t_i\in F_\U(x,z)$ for $i\leq k$.
Using Lemma \ref{small fring}.\ref{direct sum}, we can further assume the generating set for $\theta$ has the form $$\{(s_1,z),\ldots,(s_k,z),(v_1,z),\ldots,(v_k,z)\}$$ for $s_i\in R_\U$ and $v_i\in M_\U$ for $i\leq k$.

Let $I$ be the ideal of $\R_\U$ generated by $s_1,\ldots,s_k$. Let $N$ be the $\R_\U$-submodule of $\bm$ generated by $v_1,\ldots,v_k$ and $IM_\U$.
Let $\rho$ be the congruence of the $\R_\U$-module $\langle F_\U(x,z),+,R_\U\rangle$ induced by the submodule $I+N$. We want to show that $\theta=\rho$, so that $z/\theta=I+N$, proving that $I=z/\theta\cap R_\U$ and $N=z/\theta\cap M_\U$.

To see that $\rho\subseteq\theta$, notice that $z/\rho=I+N\subseteq z/\theta$. Since $\Free_\U(x,z)$ is an abelian Mal'cev algebra, that implies $\rho\subseteq \theta$ by \cite[Corollary 7.7]{MR909290}.

The last thing to prove is that $\theta\subseteq \rho$. Since $\rho$ contains a generating set for $\theta$, it suffices to show that $\rho$ is a fully invariant congruence of $\Free_\U(x,z)$. For this, we use the characterization of \ref{U con is module con}.

By construction, \ref{I}-\ref{IM} hold for $\rho$.

To see that property (e) holds, suppose that $f(z)\in N$.
Then
\begin{align*}
    f(z)    &=\sum_{i=1}^\ell p_i\cdot u_i(z) +\sum_{j=1}^k q_j\cdot v_j(z)
\end{align*}
for some $p_i\in I$ and $q_j\in R_\U$.
We have
\begin{align*}
    f(x)-f(z)-x  &=\sum_{i=1}^\ell p_i(u_i(x),x) +\sum_{j=1}^k q_j(v_j(x),x)-(k+\ell-1)x-f(z)-x\\
                &=\sum_{i=1}^\ell \bigl(p_iu_i(x)+(1-p_i)x\bigl)+\sum_{j=1}^k\bigl(q_jv_j(x)+(1-q_j)x\bigl)-(k+\ell)x-f(z)\\
                &=\sum_{i=1}^\ell p_i(u_i(x)-x)+\sum_{j=1}^kq_j(v_j(x)-x)-\sum_{i=1}^\ell p_iu_i(z)-\sum_{j=1}^kq_jv_j(z)\\
                &=\sum_{i=1}^\ell \underbrace{p_i}_{\in I}\underbrace{(u_i(x)-u_i(z)-x)}_{\in R_\U}+\sum_{j=1}^k\underbrace{q_j}_{\in R_\U}\underbrace{(v_j(x)-v_j(z)-x)}_{\in I}.
\end{align*}
Hence $f(x)-f(z)-x\in I$. This proves condition (e).

Properties \ref{I}-\ref{f(x)-f(z)} hold. So $\rho$ is a fully invariant congruence and $\theta\subseteq\rho$. This completes the proof that $\rho=\theta$. 
That means that $z/\theta\cap R_\U=I$, a finitely generated $\R_\U$ ideal, and $z/\theta\cap M_\U=N$, a finitely generated $\R_\U$-submodule of $\bm$. We have proven the forward direction of \ref{finite congruence}.

Now we show the converse.
Assume that $\theta$ is a fully invariant congruence.
For $o\in\mathbb{N}$, let $s_1,\ldots,s_o\in R_\U$ generate the ideal $I=z/\theta\cap R_\U$ of $\R_\U$. For $k\in\mathbb{N}$, let $v_1,\ldots,v_k\in M_\U$ such that $v_1+IM_\U,\ldots,v_k+IM_\U$ generate the $R_\U$-module $N/IM_\U$ for $N=z/\theta\cap M_\U$.

Let $\rho$ be the fully invariant congruence of $\Free_\U(x,z)$ generated by $$\{(s_1,z),\ldots,(s_o,z),(v_1,z),\ldots,(v_k,z)\}.$$

The claim is that $z/\rho=I+N=z/\theta$. If this is true, then $\rho=\theta$ by \cite[Corollary 7.7]{MR909290} since $\Free_\U(x,z)$ is abelian.

Since the generators of $\rho$ are in $\theta$, we have $\rho\subseteq\theta$ and hence $z/\rho\subseteq I+N.$

Then what is left to show is that $I+N\subseteq z/\rho$. First we show $I\subseteq z/\rho$. Let $t\in I$. We want $t\ \rho\ z$. We have for some $r\in\N$ some $i_1,\ldots,i_r\in\{1,\ldots,o\}$ and $p_1,\ldots,p_r,q_1,\ldots,q_r\in R$ such that 
\begin{align*}
    t=\sum_{j=1}^r p_js_{i_j}q_j.
\end{align*}
We have 
\begin{align*}
    p_js_{i_j}q_j\ &\ \rho\ p_j(z,z) \text{ because } \rho \text{ is fully invariant} \\
                    &= z,
\end{align*}
 for all $j\leq r$. Since addition is defined using $m$ and $\rho$ is a congruence, we have that $t\ \rho\ z$. So we have that $I\subseteq z/\rho$.

Now we show $N\subseteq z/\rho$. Suppose that $w\in N$. Then it must be that
\begin{align*}
    w(z)    &=  \sum_{i=1}^\ell p_i\cdot u_i+\sum_{j=1}^k q_j\cdot v_j\\
            &= \sum_{i=1}^\ell p_i(u_i,z)+\sum_{j=1}^k q_j(v_j(z),z)
\end{align*}
for some $p_i\in I$ and $q_j\in R_\U.$
We have 
\begin{align*}
    q_j(v_j(z),z)   &\ \rho\ q_j(z,z) \text{ because } \rho \text{ is a congruence }\\
                    &=z
\end{align*}
for all $j\leq k$.
We also have that
\begin{align*}
    p_i(u_i(z),z)\ \rho\ z \text{ }
\end{align*}
since $p_i\in I$ and $p_i\ \rho\ z$ and $\rho$ is fully invariant.
Since addition is defined using $m$ and $\rho$ is a congruence, we have that $w\ \rho\ z$. So we have that $N\subseteq z/\rho$.

We have shown that $I\subseteq z/\rho$ and $N\subseteq z/\rho$ to prove that $I+N\subseteq z/\rho$. Together with $z/\rho\subseteq I+N$ we have that $I+N=z/\rho$. Since $\Free_\U(x,z)$ is an abelian Mal'cev variety, this proves that $\theta=\rho$. So $\theta$ is finitely generated as a fully invariant congruence and we are done with the proof.
\end{proof}

\section{Equational Theories and Fully Invariant Congruences}
    In this section, we establish the relationship between an abelian Mal'cev variety's equational theory and a fully invariant congruence of $\Free_\U(x,z)$. We first give some needed definitions and background. Then, we state a well-known result relating the equational theories of arbitrary varieties and the fully invariant congruences of their term algebras.
    
    For a given variety $\V$, we use $\Eq(\V)$ to denote the \emph{equational theory of $\V$}. Formally, $$\Eq(\V):=\{s\approx t\ |\ \V\models s\approx t\}. $$
    For $X$ a set of variables and $\sigma$ a type, let $\bt(X)$ denote the \emph{term algebra of type $\sigma$ over $X$}. 
    Let $\Sigma$ be a set of identities over a type $\sigma$. For $\sigma$ terms $s,t$ we write \emph{$\Sigma\models s\approx t$} if whenever an algebra $\A$ satisfies every equation in $\Sigma$, then $\A\models s\approx t$. 
    
    Let $\A$ be an algebra and $S\subseteq A^2.$ Then let $\Cg_\fuil(S)$ denote the smallest fully invariant congruence of $\A$ containing $S$. We call this the \emph{fully invariant congruence generated by $S$.} Let $\Con_\fuil(\A)$ denote the lattice of fully invariant congruences of $\A$.

\begin{theorem}\label{BS1412}\cite[Chapter II, Theorem 14.12]{MR648287}
    If $\Sigma$ is a set of identities over $X$ and $p\approx q$ an identity over $X$, then
    \begin{align*}
        \Sigma\models p\approx q \iff (p,q)\in \Cg_\fuil\bigl(\{(s,t)\ |\ s\approx t\in\Sigma\}\bigl).
    \end{align*}
\end{theorem}

    For two varieties $\V,\W$, we write \emph{$\V\leq\W$} to mean $\V$ is a subvariety of $\W$.

    We now relativize Theorem \ref{BS1412} to varieties $\W$ contained in $\U$. Instead of looking at the fully invariant congruences of a term algebra, we will be looking at the fully invariant congruences of $\Free_\U(x_1,x_2,\ldots)$ and then $\Free_\U(x,z)$. The latter one will be important since we have already related the ring ideals and submodules of $\R_\U$ and $\bm_\U$ to fully invariant congruences of $\Free_\U(x,z)$ in Lemma \ref{Fring}. This brings us one step closer to proving Theorem \ref{the theorem}.

\begin{corollary}\label{rel BS1412}
    Let $X$ be a countable set of variables and let $\U$ be a variety. Then $$\tau\colon \{\Eq(\W)\ |\ \W\subseteq\U\}\to\Con_\fuil(\Free_\U(X)),\; \Eq(\W)\mapsto \{(s^\U,t^\U)\ |\ \W\models s\approx t\},$$ is a lattice isomorphism. Also, this isomorphism restricts to a bijection between the equational theories that are finitely based relative to $\U$ and the finitely generated fully invariant congruences of $\Free_\U(X)$.
\end{corollary}
\begin{proof}
    Let $\psi\colon \mathbf{T}(X)\to\Free_\U(X),\ s\mapsto s^\U$. By the Isomorphism Theorem, we know that $\mathbf{T}(X)/\ker\psi\cong\Free_\U(X)$. Also, there exists a lattice isomorphism between the fully invariant congruences of $\mathbf{T}(X)$ that contain $\ker\psi$ and the fully invariant congruences of $\Free_\U(X)$. 

    By Theorem \ref{BS1412}, there exists a lattice isomorphism between the fully invariant congruences of $\mathbf{T}(X)$ containing $\ker\psi$ and $\{\Eq(\W)\ |\ \W\leq\U\}.$ This is because the equational theory of $\W$ contains the equational theory of $\U$.

    Hence, there exists a lattice isomorphism between the fully invariant congruences of $\Free_\U(X)$ and the equational theories of subvarieties of $\U$. It is easy to check that $\tau$ is this lattice isomorphism.

    Now $\tau$ will map compact elements to compact elements. That is, $\tau$ will map finitely based (relative to $\U$) equational theories to finitely generated fully invariant congruences of $\Free_\U(X)$.
\end{proof}

\begin{lemma}\label{fin basis fin con}
    Let $\U$ be as in Definition \ref{C} and $\W$ be a subvariety of $\U$. Let $\phi\colon\Free_\U(x_1,z)\to\Free_\W(x_1,z),$ $s^\U\mapsto s^\W.$ Then $\W$ is finitely based over $\U$ if and only if $\theta:=\ker\phi$ is finitely generated as a fully invariant congruence of $\Free_\U(x_1,z)$.
\end{lemma}
\begin{proof}
    Let $X=\{z,x_1,x_2,\ldots\}$. Let $\theta_\W:=\tau(\Eq(\W)),$ where $\tau$ is the lattice isomorphism from Corollary \ref{rel BS1412}.  The claim is that $\theta$ is finitely generated as a fully invariant congruence of $\Free_\U(x,z)$ if and only if $\theta_\W$ is finitely generated as a fully invariant congruence of $\Free_\U(X)$. The latter is equivalent to $\W$ being finitely based over $\U$ by Corollary \ref{rel BS1412}. Thus, we will have proven the lemma if we can prove the claim.

    For a term $t(x_1,\ldots,x_k,z)$, let $t_0(z):=t(z,\ldots,z)$ and $t_i(x,z):=t(z,\ldots,z,x,z,\ldots,z)$ with $x$ in position $i$ for $1\leq i\leq k$.
    For terms $s,t$ we have $$\W\models s(x_1,\ldots,x_k,z)\approx t(x_1,\ldots,x_k,z)$$ if and only if, by Lemma \ref{basis lemma}, $$\W\models s_i\approx t_i,\text{ for all } 0\leq i\leq k.$$ Since $\theta=\theta_\W\cap F_\U(x_1,z)^2$, it follows that $\theta_\W=\Cg_\fuil(\theta).$

    So if $\theta$ is finitely generated as a fully invariant congruence, then so is $\theta_\W$. Conversely, if $\theta_\W$ is finitely generated, then $\theta_\W$ is generated by a finite subset $B$ of $\theta$. Then $B$ is also a finite generating set of $\theta$.

    We have shown that $\theta$ is finitely generated as a fully invariant congruence if and only if $\theta_\W$ is finitely generated as a fully invariant congruence. The latter is equivalent to $\W$ being finitely based over $\U$ by Corollary \ref{rel BS1412}. So we have proven the lemma.
    \end{proof}

\section{Proof of Main Theorem}\label{final proof}

We are now ready to prove Theorem \ref{the theorem}
        \begin{proof}[Proof of Theorem ~\ref{the theorem}]
            Let $\V$ be an abelian Mal'cev variety. 
            By Lemma \ref{term equivalence}, it suffices to consider the variety $\W\leq\U$.

            Let $\R_\U:=\langle F_\U^\id(x,z),+,\cdot\rangle$ be the ring described in Lemma \ref{small fring}.
            By Lemma \ref{Fring}.\ref{free ring}, $\R_\U$ is a free ring with generators $\unaryoperations.$
            Let $\phi\colon\Free_\U(x,z)\to \Free_\W(x,z)$ be the homomorphism defined by $t^\U(x,z)\mapsto t^\W(x,z)$ and let $\theta:=\ker\phi$.
            Let $I:=z/\theta\cap F_\U^\id(x,z)$.

            Let $\bm_\U:=\langle F_\U(z),+,R_\U\rangle$ be the $\R_\U$-module described in Lemma \ref{small fring}.
            By Lemma \ref{Fring}.\ref{free submodule}, $\bm_\U$ is a free $\R_\U$-module over $\yous.$
            Let $N:=z/\theta\cap F_\U(z)$.

            We have that $\W$ is finitely based if and only if $\theta$ is finitely generated as a fully invariant congruence of $\Free_\U(x,z)$ by Lemma \ref{fin basis fin con}. By Lemma \ref{Fring}.\ref{finite congruence}, the latter is equivalent to $I$ being a finitely generated ideal of $\R_\U$ and $N/IM_\U$ being a finitely generated $\R_\U$-module. This is further equivalent to $\R_\W\cong \R_\U/I$ being a finitely presented ring and $\bm_\W$ being a finitely presented $\R_\U$-module, equivalently a finitely presented module over $\R_\W\cong\R_\U/I$.
        \end{proof}
        \bibliographystyle{plain}
        \bibliography{bibly}

\begin{thebibliography}{1}

\bibitem{MR648287}
Stanley Burris and H.~P. Sankappanavar.
\newblock {\em A course in universal algebra}, volume~78 of {\em Graduate Texts
  in Mathematics}.
\newblock Springer-Verlag, New York-Berlin, 1981.

\bibitem{MR909290}
Ralph Freese and Ralph McKenzie.
\newblock {\em Commutator theory for congruence modular varieties}, volume 125
  of {\em London Mathematical Society Lecture Note Series}.
\newblock Cambridge University Press, Cambridge, 1987.

\bibitem{MR3793673}
Ralph~N. McKenzie, George~F. McNulty, and Walter~F. Taylor.
\newblock {\em Algebras, lattices, varieties. {V}ol. 1}.
\newblock AMS Chelsea Publishing/American Mathematical Society, Providence, RI,
  2018.
\newblock Reprint of [MR0883644], \copyright 1969.

\bibitem{MR215899}
Hanna Neumann.
\newblock {\em Varieties of groups}.
\newblock Springer-Verlag New York, Inc., New York, 1967.

\end{thebibliography}
\end{document}